\documentclass[12pt,reqno]{amsart}

\usepackage{a4wide}

\pdfoutput=1

\usepackage{amssymb}

\usepackage[all]{xy}
\usepackage{tikz}
\usetikzlibrary{arrows}
\usepackage{hyperref}
\usepackage{enumitem}
\setlist{topsep=4pt, parsep=3pt}
\hypersetup{
	colorlinks=true,
	linkcolor=blue,
	filecolor=blue,      
	urlcolor=blue,
	citecolor=blue,
}

\usepackage{xpatch}
\xapptocmd\normalsize{%
 \abovedisplayskip=10pt plus 1pt minus 3pt
 \abovedisplayshortskip=1pt plus 3pt
 \belowdisplayskip=10pt plus 2pt minus 3pt
 \belowdisplayshortskip=8pt plus 3pt minus 2pt
}{}{}

\newtheorem{theorem}{Theorem}
\newtheorem{lemma}[theorem]{Lemma}
\newtheorem{corollary}[theorem]{Corollary}
\newtheorem{proposition}[theorem]{Proposition}
\newtheorem*{maintheorem}{Main Theorem}

\theoremstyle{definition}

\theoremstyle{remark}

\newcommand{\set}[1]{\hspace{-0.8pt} \left \{ \hspace{0.03em} {#1} \hspace{0.03em} \right \}}

\begin{document}
	
	\title{Sections of polar actions}
	
	
	\title{Sections of polar actions}

    \author{Juan Manuel Lorenzo-Naveiro, Ivan Solonenko}
    
    \address{Department of Mathematics, King's College London, United Kingdom}
    \email{ivan.solonenko@kcl.ac.uk}
    
    \address{Department of Mathematics, University of Santiago de Compostela, Spain}
    \email{jm.lorenzo@usc.es}
	
    \begin{abstract}
    In this short note we provide an elementary proof of the folklore result in the theory of isometric Lie group actions on Riemannian manifolds asserting that sections of polar actions are totally geodesic.
    \end{abstract}
	
	\maketitle
	
	Let $G$ be a connected Lie group acting properly and isometrically on a complete Riemannian manifold $M$. A complete immersed submanifold $\Sigma\subseteq M$ is called a \textit{section} of the action $G \curvearrowright M$ if it meets all the orbits perpendicularly. More formally:
	\begin{enumerate}
		\setlength{\itemsep}{0pt}
		\item[$\bullet$] $\Sigma\cap (G\cdot p) \neq \varnothing$ for all $p\in M$.
		\item[$\bullet$] $T_{p}\Sigma$ and $T_{p}(G\cdot p)$ are orthogonal subspaces of $T_p M$ for all $p\in \Sigma$.
	\end{enumerate}
	
	An action is called \textit{polar} if it admits a section. Quite often, the hypotheses on $\Sigma$ are strengthened by assuming that it is embedded and/or closed. For the sake of generality, we will not make those assumptions. The aim of this note is to provide an elementary proof of the following well-known result:
	
	\begin{maintheorem}
		Let $\Sigma\subseteq M$ be a section of a polar action $G\curvearrowright M$. Then $\Sigma$ is a totally geodesic submanifold of $M$.
	\end{maintheorem}
	
	Although this is a very basic fact in the theory of isometric actions, the literature on the topic does not appear to be teeming with its various proofs. Some of the classical papers that are usually cited as containing a proof of this result are \cite{PT} by Palais and Terng (as well as their subsequent book \cite{PTbook}) and \cite{Sz} by Szenthe. The authors of these works only prove that a section $\Sigma$ is totally geodesic at those points which are regular in $M$ (i.e. whose $G$-orbits are principal). One would then expect $\Sigma$ to be globally totally geodesic because the regular points form an open and dense subset $M_\mathrm{reg}$ of $M$ and the second fundamental form of $\Sigma$ is a smooth tensor field. The problem here is that $\Sigma \cap M_\mathrm{reg}$ may not, in theory, be dense in $\Sigma$. While some authors (c.f. \cite[Exercise 4.9 (iii)]{AB}, \cite[Proposition 1.3 (a)]{KG}, \cite[Theorem 30.9 (4)]{Mi}) comment briefly on why this statement should hold, it appears that no detailed proofs are currently available. In this note, we follow a route proposed in \cite{Be} and give a complete proof that sections are always totally geodesic. In particular, we show that $\Sigma \cap M_\mathrm{reg}$ is indeed dense in $\Sigma$.
	
	We will use a number of standard facts about proper isometric actions:
	\begin{enumerate}
	    \item \cite[Definition 29.2]{Mi} For any $p \in M$, there exists $\varepsilon > 0$ such that the restriction of the normal exponential map $\exp^\perp \colon \nu(G \cdot p) \to M$ to $\nu^\varepsilon(G \cdot p) = \set{\xi \in \nu(G \cdot p) \mid ||\xi|| < \varepsilon}$ is a diffeomorphism onto an open neighborhood of $G \cdot p$, which we denote $U^\varepsilon(G \cdot p)$ (here $\nu(G \cdot p)$ stands for the normal bundle of the orbit $G \cdot p$). This is a version of the tubular neighborhood theorem adapted to an extrinsically homogeneous submanifold. For any such $\varepsilon$, the submanifold $S_p^\varepsilon = \exp^\perp(\nu_p^\varepsilon(G \cdot p))$ is a slice of the action at $p$. Note that $S_p^\varepsilon \cap S_q^\varepsilon = \varnothing$ for any point $q \in G \cdot p$ other than $p$.
	    \item \cite[Corollary 2.2.2]{Be} An orbit $G \cdot p$ is principal if and only if the slice representation $G_p \to \mathrm{O}(\nu_{p}(G\cdot p)), g \mapsto g_{*p}$, is trivial.
	    \item \cite[Theorem 29.14]{Mi} Regular points of $G \curvearrowright M$ form an open and dense subset of $M$.
	    \item \cite[Section 2.3.1]{Be} If the action of $G$ is polar and $\Sigma$ is a section, then $\dim \Sigma$ coincides with the cohomogeneity of the action, which we denote by $\text{cohom}(G\curvearrowright M)$. In particular, if $G \cdot p$ is a principal or exceptional orbit, we have $T_p \Sigma = \nu_p(G \cdot p)$.
	\end{enumerate}

    Whenever we write $S_p^\varepsilon$ or $U^\varepsilon(G \cdot p)$, we assume that $\varepsilon$ is small enough as in (1) above. We start by recalling the proof of the fact that sections are totally geodesic at their regular points.
	\begin{proposition}[{\cite[Theorem 2.3.2]{Be}}]\label{II}
		If $\Sigma$ is a section of a polar action $G\curvearrowright M$, then its second fundamental form vanishes at all the regular points of $M$ lying in $\Sigma$.
	\end{proposition}
	\begin{proof}
		Let $p\in \Sigma$ be any point such that $G\cdot p$ has maximal dimension. For any $\xi \in \nu_p \Sigma = T_{p}(G\cdot p)$, we can find an element $X$ in the Lie algebra $\mathfrak{g}$ of $G$ such that
		$$
		\xi=X^*_p, \; \text{where} \hspace{0.4em} X^*_q = \dfrac{d}{dt}\bigg|_{t=0} \exp(tX)\cdot q.
		$$
		Here, $\exp$ denotes the Lie exponential map of $G$. Now, $X^{*}$ is a Killing vector field, so its covariant derivative is skew-symmetric. Moreover, $X^*$ is everywhere orthogonal to $\Sigma$ by polarity. Let $\mathbb{II}$ be the second fundamental form of $\Sigma$. For any $v\in T_{p}\Sigma$, we have $\langle \mathbb{II}(v,v),\xi \rangle=-\langle v,\nabla_{v}X^{*} \rangle=0$. Since $\mathbb{II}(v,v)$ is tangent to $G \cdot p$, we obtain $\mathbb{II}(v,v)=0$. Polarizing, we get $\mathbb{II}=0$. Note that we have actually proved that $\mathbb{II}$ vanishes not only at the regular points of $\Sigma$, but also at those whose orbits are exceptional. \qedhere		
	\end{proof}
	
	The key to finishing the proof is to see that, in fact, regular points are dense in any section. Since the second fundamental form is smooth, the preceding proposition implies that it will then vanish globally. We will need a few lemmas.
	
	\begin{lemma}\label{isotropies}
		Let $p\in M$ be arbitrary. Then $G_{q}\subseteq G_{p}$ for all $q\in S^{\varepsilon}_{p}$.
	\end{lemma}
	\begin{proof}[Proof]
		For any $g\in G_{q}$, we have $q=g\cdot q\in S^{\varepsilon}_{p}\cap S^{\varepsilon}_{g\cdot p} \ne \varnothing$, which implies $g\cdot p = p$ and therefore $g\in G_{p}$. \qedhere
	\end{proof}
	
	\begin{corollary}\label{corollary}
	Every $p \in M$ has an open neighborhood $U$ such that the orbit type of any $q \in U$ is greater than or equal to that of $p$, i.e. $G_q$ is conjugated to a subgroup of $G_p$.
	\end{corollary}
	
	\begin{proof}
	    Just take $U = U^\varepsilon(G \cdot p) = G \cdot S_p^\varepsilon$. Since points on the same orbit have the same orbit type, everything follows from Lemma \ref{isotropies}.
	\end{proof}
	
	\begin{lemma}[{\cite[Exercise 2.11.4]{Be}}]
	    For each $p \in M$, the cohomogeneity of the slice representation at $p$ coincides with the cohomogeneity of the action $G \curvearrowright M$.
	\end{lemma}
	
	\begin{proof}
	Let $\varepsilon$ be small enough as before and let $\xi \in \nu_p^\varepsilon(G \cdot p)$. We begin by proving that $(G_p)_\xi = G_{\exp(\xi)}$. The diffeomorphism $\exp^\perp \colon \nu_p^\varepsilon(G \cdot p) \to U_p^\varepsilon(G \cdot p) = S_p^\varepsilon$ is $G_p$-equivariant, which implies that $(G_p)_\xi = (G_p)_{\exp(\xi)} \subseteq G_{\exp(\xi)}$. But now it follows from Lemma \ref{isotropies} that $G_{\exp(\xi)} \subseteq G_p$, so $(G_p)_{\exp(\xi)}$ simply coincides with $G_{\exp(\xi)}$. We  thus have:
	\begin{align*}
	\mathrm{codim}_M(G \cdot \exp(\xi)) &= \dim M - (\dim G - \dim G_{\exp(\xi)}) \\
	&= \dim M - (\dim G - \dim G_p + \dim G_p - \dim (G_p)_\xi) \\
	&= (\dim M - \dim (G \cdot p)) - \dim (G_p \cdot \xi) \\
	&= \mathrm{codim}_{\nu_p(G \cdot p)}(G_p \cdot \xi).
	\end{align*}
	Consequently, $\mathrm{cohom}(G_p \curvearrowright \nu_p(G \cdot p)) = \mathrm{cohom}(G \curvearrowright U^\varepsilon(G \cdot p))$. But since the regular points of $G \curvearrowright M$ are dense in $M$, there is one in $U^\varepsilon(G \cdot p)$, hence $\mathrm{cohom}(G \curvearrowright U^\varepsilon(G \cdot p)) = \mathrm{cohom}(G \curvearrowright M)$.
	\end{proof}
	
	\begin{lemma}[{\cite[Exercise 2.11.5]{Be}}]\label{invariants}
		Let $p\in M$ be such that $G\cdot p$ is a nonprincipal orbit, and let $V \subseteq \nu_{p}(G\cdot p)$ be the subspace of fixed points for the slice representation at $p$. Then $\dim V$ is strictly smaller than the cohomogeneity of the action $G\curvearrowright M$.
	\end{lemma}
	\begin{proof}
		Let $x\in \nu_{p}(G\cdot p)$ be a regular point of the slice representation. Write it as $x=y+z$, where $y\in V$ and $z\in V^{\perp}$. Since the slice representation is orthogonal, it leaves $V^\perp$ invariant. Moreover, $V^\perp$ is a subspace of positive dimension because the slice representation at $p$ in nontrivial by assumption. The codimension of $G_{p}\cdot x$ in $\nu_{p}(G\cdot p)$ is equal to the cohomogeneity of the slice representation, which is also the cohomogeneity of the action. Notice that $G_{p}\cdot x=y+G_{p}\cdot z$,
		which means that $G_{p}\cdot x$ and $G_{p}\cdot z$ have the same dimension. Since $G_p \cdot z$ is compact and contained in $V^\perp$, it must be of positive codimension in $V^\perp$. Altogether, we conclude:
		\begin{align*}
		\mathrm{cohom}(G \curvearrowright M) &= \mathrm{codim}_{\nu_p(G \cdot p)}(G_p \cdot x) \\
		&= \mathrm{codim}_{\nu_p(G \cdot p)}(G_p \cdot z) \\
		&= \dim V + \mathrm{codim}_{V^\perp}(G_p \cdot z) > \dim V,
		\end{align*}
		finishing the proof.
	\end{proof}
	
	\begin{lemma}
		Assume that $\Sigma$ is a section whose set of regular points is not dense. Then there is a nonempty open subset $\Omega \subseteq \Sigma$ such that all of its points have the same orbit type but none of them is regular.
		\label{lemma:OpenSubsetsWithSameType}
	\end{lemma}
	\begin{proof}
		By our assumption, there must be a nonempty open subset $\Xi \subseteq \Sigma$ containing no regular points. 
		Let $\Xi' \subseteq \Xi$ stand for the subset of points whose isotropy subgroups have the smallest possible dimension among all points of $\Xi$. Pick $p \in \Xi'$ such that $G_p$ has the smallest possible number of connected components among all points of $\Xi'$ (this is possible because all the isotropy subgroups of $G$ are compact).
		By Corollary \ref{corollary}, there is a neighborhood $U$ of $p$ in $M$ such that the orbit type of every $q \in U$ is greater than or equal to that of $p$. By construction, $\Omega = U \cap \Xi \subseteq \Sigma$ consists of nonregular points of the same orbit type.
	\end{proof}
	\vspace{-0.1em}
	Now we are ready to prove the main result.
	\vspace{-0.1em}
	\begin{proof}[Proof of the main theorem]
		By Proposition \ref{II}, it suffices to prove that $M_\mathrm{reg} \cap \Sigma$ is dense in $\Sigma$. Suppose this is not the case. Then by Lemma~\ref{lemma:OpenSubsetsWithSameType}, there exists a nonempty open subset $\Omega\subseteq \Sigma$ all of whose points have the same nonprincipal orbit type. Pick any $p\in \Omega$, take $\varepsilon > 0$ small enough as before, and denote $S=S_{p}^{\varepsilon}$. Shrinking $\Omega$ if necessary, we may assume that it is an embedded submanifold of $M$ and lies in $U^\varepsilon(G \cdot p)$.
		
		Let $F \colon G \times S \twoheadrightarrow U^{\varepsilon}(G\cdot p)$ be the smooth map defined by $F(g,x)=g\cdot x$. Since the orbits in $U^\varepsilon(G \cdot p)$ intersect the slices $S_q^\varepsilon \; (q \in G \cdot p)$ transversely, $F$ is a smooth submersion. Given $(g,s), (g',s') \in G \times S$, one has $F(g, s) = F(g', s')$ if and only if $s = g^{-1} g' \cdot s'$. Since the latter lies both in $S$ and $g^{-1} g' \cdot S$, we deduce that $g^{-1} g' \in G_p$. It follows that the fibers of $F$ coincide with the orbits of the action
		\begin{equation*}
			G_{p}\curvearrowright G\times S, \quad h\cdot (g,s)=(gh^{-1},h\cdot s),
		\end{equation*}
		so $F$ induces a diffeomorphism\footnote{Note that $G\times_{G_{p}}S$ is the fiber bundle over $G/G_p \cong G \cdot p$ with a model fiber $S$ associated to the principal $G_p$-bundle $G \twoheadrightarrow G/G_p \cong G \cdot p$.} between $G\times_{G_{p}}S=(G\times S)/G_{p}$ and $U^{\varepsilon}(G\cdot p)$. We now consider the restriction $\widetilde{F}\colon F^{-1}(\Omega)\twoheadrightarrow \Omega$, which is still a submersion. We claim that the smooth map $\beta\colon F^{-1}(\Omega)\to S$ given by $\beta(g,s)=s$ is constant along the fibers of $\widetilde{F}$. Indeed, consider elements $(g,s),(g',s')\in F^{-1}(\Omega)$ such that $g\cdot s=g'\cdot s'$. Since $g\cdot s\in \Omega$, the points $s$ and $p$ have the same orbit type, and since $G_{s}\subseteq G_{p}$, we conclude that $G_{s}=G_{p}$. Furthermore, we know that there exists some $h\in G_{p}=G_{s}$ such that $(g',s')=(gh^{-1},h\cdot s)=(gh^{-1},s)$, so $s=s'$ as claimed. For this reason, $\beta$ factors through $\widetilde{F}$ to a smooth map $f\colon \Omega \to S$ such that $f(\widetilde{F}(g,s))=s$.
		
		We assert that $f$ is an injective immersion. It suffices to prove that it is an immersion at $p$, since shrinking $\Omega$ once again would let us assume it is an injective immersion globally. In order to see this, we will show that $f_{*p}$ is just the inclusion $T_p\Omega = T_p\Sigma \hookrightarrow \nu_p(G \cdot p) = T_pS$. Indeed, let $\pi \colon G \to G \cdot p$ be the orbit map at $p$, and let $\mathfrak{g}$ denote the Lie algebra of $G$. Consider the following commutative diagram:
		$$
	    \xymatrix{
				T_{(e,p)}(F^{-1}(\Omega)) \ar@{->>}[d]_{\widetilde{F}_{*(e,p)}} \ar[dr]^(0.6){\beta_{*(e,p)}} \\
				T_{p}\Omega \ar[r]_{f_{*p}} & T_{p}S
	    }
		$$
		Observe that $F_{*(e,p)} \colon \mathfrak{g} \oplus T_pS \to T_pM = T_p(G \cdot p) \oplus T_pS$ decomposes as $\pi_{*e} \colon \mathfrak{g} \to T_p(G \cdot p)$ and $\mathrm{Id} \colon T_pS \to T_pS$. Given $X \in T_p\Omega$, there is some $Y \in T_{(e,p)}(F^{-1}(\Omega))$ such that $X = \widetilde{F}_{*(e,p)}(Y)$. This $Y$ must be of the form $X' + X$, where $X' \in \mathfrak{g}$. But since $\beta$ is just the projection onto the second factor of $G \times S$, we have $f_{*p}(X) = \beta_{*(e,p)}(X' + X) = X$.
		
		To conclude, consider the submanifold $\widetilde{\Omega}=f(\Omega)\subseteq S$. A point of $\widetilde{\Omega}$ is of the form $f(g\cdot s)=s$, where $(g,s)\in G\times S$ is such that $g\cdot s\in \Omega$. As we saw earlier, $G_{s}=G_{p}$, so $G_{p}\cdot s=\{s\}$. This means that $\widetilde{\Omega}$ is pointwise fixed by $G_{p}$, hence so is $T_{p}\widetilde{\Omega} \subseteq \nu_{p}(G\cdot p)$. Since
		$$
		\dim T_{p}\widetilde{\Omega} = \dim \widetilde{\Omega} = \dim \Omega = \dim \Sigma = \mathrm{cohom}(G \curvearrowright M),
		$$
		we arrive at a contradiction by Lemma \ref{invariants}.
	\end{proof}
	
	\bibliographystyle{amsplain}
	
\end{document}